\documentclass[reqno]{amsart}
\usepackage{amsfonts,amssymb,latexsym,amsmath,amsxtra,graphicx,url}
\usepackage{verbatim}

\newtheorem{theorem}{Theorem}

\newcommand{\N}{\mathbb{N}}

\title{The method of weighted words revisited}
\author{Jehanne Dousse}
\address{Institut f\"ur Mathematik, Universit\"at Z\"urich\\ Winterthurerstrasse 190, 8057 Z\"urich, Switzerland}
\email{jehanne.dousse@math.uzh.ch}

\begin{document}

\begin{abstract}
Alladi and Gordon introduced the method of weighted words in 1993 to prove a refinement and generalisation of Schur's partition identity. Together with Andrews, they later used it to refine Capparelli's and G\"ollnitz' identities too. In this paper, we present a new variant of this method, which can be used to study more complicated partition identities, and apply it to prove refinements and generalisations of three partition identities. The first one, Siladi\'c's theorem (2002), comes from vertex operator algebras. The second one, a conjectural identity of Primc (1999), comes from crystal base theory. The last one is a very general identity about coloured overpartitions which generalises and unifies several generalisations of Schur's theorem due to Alladi-Gordon, Andrews, Corteel-Lovejoy, Lovejoy and the author.
\end{abstract}

\maketitle

\section{Introduction and principle of the method}
\subsection{The original method}
A partition of $n$ is a non-increasing sequence of positive integers whose sum is $n$. For example, the $5$ partitions of $4$ are $4$, $3+1$, $2+2$, $2+1+1$ and $1+1+1+1$.

The method of weighted words was introduced by Alladi and Gordon~\cite{Alladi1} to give refinement of Rogers-Ramanujan type partition identities, which are theorems of the form ``for all $n$, the number of partitions of $n$ whose parts satisfy some difference conditions is equal to the number of partitions of $n$ whose parts satisfy some congruence conditions.''
The first identity they applied it to in~\cite{Alladi,AllGor} was Schur's theorem~\cite{Schur}:
\begin{theorem}[Schur]
\label{schur}
For any integer $n$, let $A(n)$ denote the number of partitions of $n$ into distinct parts congruent to $1$ or $2$ modulo $3$ and $B(n)$ the number of partitions of $n$ such that parts differ by at least $3$ and no two consecutive multiples of $3$ appear. Then for all $n$, $A(n)=B(n).$
\end{theorem}

The idea of the method of weighted words is to work at the ``non-dilated level'' and consider partitions into integers appearing in different colours $a$,$b$,$c$,... which, under the correct dilations, become the partitions of the considered theorem. The colours also represent free parameters which allow one to refine the theorems.

For example, in the case of Schur's theorem, Alladi and Gordon used three colours
\begin{equation} \label{colororder}
c < a < b,
\end{equation}
giving the order $ 1_c < 1_a < 1_b < 2_{c} < 2_a < 2_b < \cdots$ on coloured integers.
They considered partitions $\lambda_1 + \cdots + \lambda_s$ into coloured integers, with no part $1_c$, satisfying the difference conditions
$$\lambda_i - \lambda_{i+1} \geq \begin{cases} 2 \text{ if } c(\lambda_{i}) = ab \text{ or } c(\lambda_i) < c(\lambda_{i+1}) \text{ in \eqref{colororder}}\\
1 \text{ otherwise,}
\end{cases}$$
where $c(\lambda)$ denotes the colour of $\lambda$.
Under the dilations
\begin{equation}
\label{dilschur}
q \rightarrow q^3, a \rightarrow a q^{-2}, b \rightarrow bq^{-1}, c \rightarrow cq^{-3},
\end{equation}
the coloured integers are transformed as follows
$$k_a \rightarrow 3k-2, k_b \rightarrow 3k-1, k_c \rightarrow 3k-3,$$
and the partitions above become those of Schur's theorem.

Then, they try to find for which values of $a,b,c$ the generating function for these coloured partitions is an infinite product, which means that it is also the generating function for partitions with congruence conditions.

To do so, Alladi and Gordon used the fact that a partition with $n$ parts satisfying the difference conditions is a minimal partition with $n$ parts satisfying the difference conditions to which one has added a partition having at most $n$ parts. So they computed the generating function for such minimal partitions using $q$-binomial coefficients, and concluded that is an infinite product representing partitions with congruence conditions if and only if $c=ab$ by using $q$-series identities.

\begin{theorem}[Alladi-Gordon]
\label{schurnondil}
Let $A(u,v,n)$ be the number of partitions of $n$ into $u$ distinct parts coloured $a$ and $v$ distinct parts coloured $b$.
Let $B(u,v,n)$ be the number of partitions $\lambda_1 + \cdots + \lambda_s$ of $n$ into distinct parts, with no part $1_{ab}$, such that $\lambda_i - \lambda_{i+1} \geq 2$ if $c(\lambda_{i}) = ab$ or $c(\lambda_i) < c(\lambda_{i+1})$ in \eqref{colororder}, having $u$ parts $a$ or $ab$ and $v$ parts $b$ or $ab$.

\noindent
Then for all $u,v,n \in \N,$ $A(u,v,n) = B(u,v,n).$

\noindent
In other words,
$$\sum_{u,v,n \in \N} B(u,v,n) a^ub^vq^n = (-aq;q)_{\infty}(-bq;q)_{\infty}.$$
\end{theorem}
Here we used the $q$-series notation $(a;q)_{\infty} := \prod_{n \geq 0} (1-aq^n).$ Under the dilations \eqref{dilschur}, this gives a refinement of Schur's theorem keeping track of the number of parts whose colour involves $a$ and the number of parts whose colours involve $b$.
Such non-dilated theorems are very interesting, as using other dilations or a different ordering of the colours can lead to infinitely many new identities.

This method was successfully applied to other identities too, such as G\"ollnitz' or Capparelli's theorems~\cite{AllAndGor2,AllAndGor}. However, when there are too many colours or when the difference conditions are too complicated, it might be hard to apply this method directly, as the minimal partitions might be hard to compute (even in the above-mentioned papers, this required long computations involving $q$-binomial or $q$-multinomial coefficients). Moreover it might also be complicated to find helpful $q$-series identities when we are working with too many colour variables. Therefore, we introduce a new version of the method of weighted words, which doesn't use minimal partitions and $q$-series identities, but only recurrences and $q$-difference equations.

\subsection{The new version of the method}
Our method starts as the original one. We assign a colour to each residue class modulo the number considered in the partition identity. For example in the case of Schur's theorem, we will assign colour $a$ (resp. $b,c$) to the numbers congruent to $1$ (resp. $2,3$) modulo $3$. Then, before going at the non-dilated level, we consider partitions with difference conditions with those added parameters. By studying the first terms of their generating function, we find some necessary conditions on the colours to obtain an infinite product which is the generating function for partitions with congruence conditions. This can be done with a computer program when there are too many colours.
In the case of Schur's theorem, the series starts with $1+aq+bq^2+cq^3+aq^4+(a^2+b)q^5+\cdots .$ For this to be the beginning of the expansion of a suitable infinite product, we need $c=ab$.

Then, we keep these restrictions on the colours and translate the difference conditions at the non-dilated level as in Alladi and Gordon's method. Now we want to show that the generating function for these partitions is an infinite product, but avoid using minimal partitions. To do so, we define, for $q,a,b,...$ of module smaller than $1$, $G_k(q;a,b,...)$, the generating function for the coloured partitions with difference conditions with largest part $\leq k$ (where $k$ is a coloured integer), where the power of $q$ (resp. $a,b,...$) is the number partitioned (resp. the number of parts with colour $a,b,...$). Using the difference conditions, we give recurrences satisfied by those $G_k(q;a,b,...)$'s, and initial conditions chosen so that we obtain the correct first values of $G_k(q;a,b,...)$ with the recurrences.
For example, in the case of Schur's theorem, we obtain
\begin{align*}
G_{k_a}(q;a,b) &= G_{k_{ab}} (q;a,b) + aq^{k} G_{(k-1)_a}(q;a,b),\\
G_{k_b}(q;a,b) &= G_{k_a} (q;a,b) + bq^{k} G_{(k-1)_b} (q;a,b),\\
G_{k_{ab}}(q;a,b) &= G_{(k-1)_b} (q;a,b) + abq^{k} G_{(k-2)_b} (q;a,b),
\end{align*}
and the initial conditions
\begin{align*}
G_{0_a}(q;a,b)&=G_{0_b}(q;a,b)=G_{0_{ab}}(q;a,b)=1,\\
G_{-1_a}(q;a,b)&=G_{-1_b}(q;a,b)=G_{-1_{ab}}(q;a,b)=0.
\end{align*}

Finally, we use these recurrences to find $G_{\infty}(q;a,b,...):= \lim_{k \rightarrow \infty} G_{k}(q;a,b,...)$, which will be the generating function for all coloured partitions with difference conditions, as there is no more restriction on the largest part. This is the only non-automatic step, and the techniques to do it may vary. For example, in the case of Schur's theorem, we can prove that for all $k \in \N$,
$$G_{(k+2)_{ab}}(q;a,b)= (1+aq)(1+bq) G_{k_b}(q;aq,bq).$$
Therefore we have
$$G_{\infty}(q;a,b)= (1+aq)(1+bq) G_{\infty}(q;aq,bq),$$
and iterating leads to
$$G_{\infty}(q;a,b)= (-aq;q)_{\infty}(-bq;q)_{\infty}G_{\infty}(q;0,0) = (-aq;q)_{\infty}(-bq;q)_{\infty},$$
which is an infinite product as wanted. This completes the proof.

In the next section, we present three different applications of this new version of the method of weighted words.

\section{Applications}
\subsection{Siladi\'c's theorem}
The first example of application of our method is Siladi\'c's theorem~\cite{Siladic}, a partition identity which was proved in 2002 by studying level $1$ modules of the twisted affine Lie algebra $A^{(2)}_2.$

\begin{theorem}[Siladi\'c]
\label{refinement}
Let $n \in \N$. Let $A(n)$ denote the number of partitions of $n$ into distinct odd parts, and $B(n)$ denote the number of partitions $\lambda_1 + \cdots+ \lambda_s$ of $n$ satisfying the following conditions:
\begin{enumerate}
  \item $\forall i \geq 1, \lambda_i \neq 2$,
  \item $\forall i \geq 1, \lambda_i - \lambda_{i+1} \geq 5$,
  \item $\forall i \geq 1$,
  \begin{equation*}
  \begin{aligned}
&\lambda_i - \lambda_{i+1} = 5 \Rightarrow \lambda_i \equiv 1, 4 \mod 8,
\\&\lambda_i - \lambda_{i+1} = 6 \Rightarrow \lambda_i \equiv 1, 3, 5, 7 \mod 8,
\\&\lambda_i - \lambda_{i+1} = 7 \Rightarrow \lambda_i \equiv 0, 1, 3, 4, 6, 7 \mod 8,
\\&\lambda_i - \lambda_{i+1} = 8 \Rightarrow \lambda_i \equiv 0, 1, 3, 4, 5, 7 \mod 8.
	\end{aligned}
	\end{equation*}
\end{enumerate}
Then for all $n$, $A(n)=B(n)$.
\end{theorem}

Siladi\'c's theorem is a good example of an identity where the classical method of weighted words would be difficult to apply. The difference conditions are quite intricate so it seems hard to find the minimal partitions and therefore the generating function for the partitions with difference conditions. Even if that was possible, it might also not be that easy to find a $q$-series identity (with up to $8$ colour variables) which would lead to the correct infinite product.

However, our new version of the method works quite well in that case.
First, we assign a different colour to each congruence class modulo $8$. By studying the first terms of the generating function for partitions with difference conditions, we notice that some relations between colours are necessary for this generating function to be an infinite product. The integers congruent to $1$ or $5$ modulo $8$ have colour $a$, those congruent to $3$ or $7$ modulo $8$ have colour $b$, those congruent to $0$ or $4$ modulo $8$ have colour $ab$, those congruent to $2$ modulo $8$ have colour $b^2$ and those congruent to $6$ modulo $8$ have colour $a^2$, where $a$ and $b$ are some free parameters. The infinite product we seem to obtain in that case is $(-aq;q^4)_{\infty} (-bq^3;q^4)_{\infty}$.

Now let us translate this at the non-dilated level.
We consider integers appearing in five colours $a$, $b$, $ab$, $a^2$ and $b^2$, ordered as follows:
$$1_{ab} < 1_a < 1_{b^2} <1_{b} <2_{ab} < 2_a <3_{a^2} < 2_{b} <3_{ab} < 3_a < 3_{b^2} <3_b < \cdots .$$
Note that the colours $a^2$ and $b^2$ only appear for odd integers, and that $1_{a^2}$ doesn't appear.
We consider partitions $\lambda_1 + \cdots + \lambda_s$ where the entry $(x,y)$ in the matrix $A$ gives the minimal difference between $\lambda_i$ of colour $x$ and $\lambda_{i+1}$ of colour $y$:
$$
A=\bordermatrix{\text{} & a & b & ab & a^2 & b^2 \cr a_{odd} & 2&2&1&2&2 \cr b^2 &2&3&2&2&4 \cr b_{odd} &1&2&1&2&2 \cr ab_{even} &2&2&2&3&3 \cr a_{even}&2&2&2&3&3 \cr a^2 &3&3&3&4&4 \cr b_{even} &1&2&1&1&3 \cr ab_{odd} &2&3&2&2&3}.
$$
We defined this order and matrix such that under the dilations
\begin{equation}
\label{dilat}
q \rightarrow q^4, a \rightarrow aq^{-3}, b \rightarrow bq^{-1},
\end{equation}
the order on the coloured integers becomes the natural ordering
$$0_{ab} < 1_a < 2_{b^2} <3_{b} <4_{ab} < 5_a <6_{a^2} < 7_{b} <8_{ab} < 9_a < 10_{b^2} <11_b < \cdots, $$
and the difference conditions become those of Siladi\'c's theorem.

We prove the following refinement and non-dilated version of Siladi\'c's theorem.
\begin{theorem}
\label{th:nondil}
For $u,v,n \in \N$, let $D(u,v,n)$ denote the number of partitions $\lambda_1 + \cdots + \lambda_s$ of $n$, where $1$ can only be of colour $a$, satisfying the difference conditions given by the matrix $A$, such that $u$ equals the number of parts $a$ or $ab$ plus twice the number of parts $a^2$ and $v$ equals the number of parts $b$ or $ab$ plus twice the number of parts $b^2$.

Then
\begin{equation*}
\sum_{u,v,n \in \N} D(u,v,n)a^ub^vq^n = (-aq;q)_{\infty} (-bq;q)_{\infty}.
\end{equation*}
\end{theorem}
\begin{proof}[Idea of the proof of Theorem~\ref{th:nondil}]
We define $G_k(q;a,b)$ to be the generating function for coloured partitions with difference conditions and largest part at most $k$.
Using combinatorial reasoning on the largest part and the difference conditions of matrix $A$, we start by giving eight recurrences for the $G_k (q;a,b)$'s, such as
$$G_{2k+1_{ab}}(q;a,b)=G_{2k_{b}}(q;a,b)+abq^{2k+1} G_{2k-1_{a}}(q;a,b).$$
Then we use them to prove the following equations by induction on $k \in \N^*$:
\begin{align*}
G_{2k+1_{ab}}(q;a,b) &= (1+aq) G_{2k_{a}}(q;b,aq),\\
G_{2k+1_{b^2}}(q;a,b) &= (1+aq) G_{2k_{b}}(q;b,aq),\\
G_{2k+2_{ab}}(q;a,b) &= (1+aq) G_{2k+1_{a}}(q;b,aq), \\
G_{2k+1_{a^2}}(q;a,b) &= (1+aq) G_{2k-1_{b}}(q;b,aq).
\end{align*}
Finally, letting $k$ tend to infinity and iterating leads to
\begin{align*}
G_{\infty}(q;a,b) &= (1+aq) G_{\infty} (q;b,aq)\\
&= (1+aq)(1+bq) G_{\infty} (q;aq,bq)\\
&= (-aq;q)_{\infty}(-bq;q)_{\infty}.
\end{align*}
This is the generating function for partitions into distinct parts coloured $a$ or $b$.
\end{proof}

By doing the dilations~\eqref{dilat}, we obtain the following new refinement of Siladi\'c's theorem.
\begin{theorem}
\label{th:refdilat}
For $u,v,n \in \N$, let $C_4(u,v,n)$ denote the number of partitions of $n$ into $u$ distinct parts congruent to $1$ modulo $4$ and $v$ distinct parts congruent to $3$ modulo $4$.
Let $D_4(u,v,n)$ denote the number of partitions $\lambda_1 + \cdots + \lambda_s$ of $n$ such that $u$ equals the number of parts congruent to $0$ or $1$ modulo $4$ plus twice the number of parts congruent to $6$ modulo $8$ and $v$ equals the number of parts congruent to $0$ or $3$ modulo $4$ plus twice the number of parts congruent to $2$ modulo $8$, satisfying the difference conditions.
Then $C_4(u,v,n)=D_4(u,v,n)$.
\end{theorem}

Moreover, the non-dilated Theorem~\ref{th:nondil} allows one to obtain infinitely many new identities by doing different dilations. In particular, the infinite product in Theorem~\ref{th:nondil} is exactly the same as in Alladi and Gordon's non-dilated version of Schur's theorem (Theorem~\ref{schurnondil}). Thus the same dilations as theirs leads to a new companion of Schur's theorem.
\begin{theorem}
\label{newschur}
For $u,v,n \in \N$, let $C_3(u,v,n)$ denote the number of partitions of $n$ into $u$ distinct parts congruent to $1$ modulo $3$ and $v$ distinct parts congruent to $2$ modulo $3$.
Let $D_3(u,v,n)$ denote the number of partitions $\lambda_1 + \cdots + \lambda_s$ of $n$ in two colours, say ordinary and primed, such that only parts congruent to $\pm 1 \mod 6$ can be primed, $1'$ is not a part, and such that $u$ equals the number of ordinary parts congruent to $0$ or $1$ modulo $3$ plus twice the number of primed parts congruent to $5$ modulo $6$ and $v$ equals the number of ordinary parts congruent to $0$ or $2$ modulo $3$ plus twice the number of primed parts congruent to $1$ modulo $6$, satisfying the following conditions:
$$\lambda_i - \lambda_{i+1}  \geq
\begin{cases}
4 + \chi(\lambda_{i+1}') \text{ if } \lambda_{i} \equiv 1, 2, 3, 5 \mod 6 \text{ and is not primed} , \\
5 + \chi(\lambda_{i+1}') \text{ if } \lambda_{i} \equiv 0, 4 \mod 6, \\
6 + \chi(\lambda_{i+1}') \text{ if } \lambda_{i} \equiv 1, 5 \mod 6 \text{ and is primed},
\end{cases}$$
where $$\chi(\lambda_{i+1}') = 
\begin{cases}
= 1 \text{ if $\lambda_{i+1}$ is primed} , \\
= 0 \text{ otherwise}.
\end{cases}$$
Then $C_3(u,v,n)=D_3(u,v,n)$.
\end{theorem}

Detailed proofs and more applications can be found in \cite{Doussesil2}.

\subsection{Primc's conjecture}
Our method also applies to prove another result coming from representation theory. In~\cite{Primc}, Primc studied partition identities arising from crystal base theory. In particular, he considered partitions $\lambda_1 + \cdots + \lambda_s$ into integers in four colours $a,b,c,d$, with the order
\begin{equation}
\label{orderprimc}
1_{a} < 1_{b} < 1_{c} <1_{d} <2_{a} < 2_{b} <2_{c} < 2_{d} < \cdots ,
\end{equation}
such that the entry $(x,y)$ in the matrix $B$ gives the minimal difference between $\lambda_i$ of colour $x$ and $\lambda_{i+1}$ of colour $y$:
$$
B=\bordermatrix{\text{} & a & b & c & d \cr a & 2&1&2&2 \cr b &1&0&1&1 \cr c &0&1&0&2 \cr d&0&1&0&2}.
$$
Primc conjectured that, under the dilations
$$k_{a} \rightarrow 2k-1, k_{b} \rightarrow 2k, k_{c} \rightarrow 2k, k_{d} \rightarrow 2k+1,$$
the generating function for these coloured partitions is equal to $\frac{1}{(q;q)_{\infty}}.$

We can use our method to prove Primc's conjecture, and actually refine it as a non-dilated partition identity where we keep track of the parts coloured $a$, $c$ and $d$ (the usual test with the first values of the generating function shows that we should set $b=1$ to obtain a suitable infinite product).

\begin{theorem}
\label{th:main}
Let $A(n;k,\ell,m)$ denote the number of partitions defined above with $k$ parts coloured $a$, $\ell$ parts coloured $c$ and $m$ parts coloured $d$. Then
$$\sum_{n,k,\ell,m \geq 0} A(n;k,\ell,m) q^n a^k c^{\ell} d^m = \frac{(-aq;q^2)_{\infty}(-dq;q^2)_{\infty}}{(q;q)_{\infty}(cq;q^2)_{\infty}}.$$
\end{theorem}

Under the dilations
$$q \rightarrow q^2, a \rightarrow aq^{-1}, c \rightarrow c, d \rightarrow dq,$$
the ordering of integers~\eqref{orderprimc} becomes
$$1_a < 2_b <2_c < 3_d <3_a <4_b <4_c<5_d < \cdots ,$$
the matrix $B$ becomes
$$
B_2=\bordermatrix{\text{} & a & b & c & d \cr a & 4&1&3&2 \cr b &3&0&2&1 \cr c &1&2&0&3 \cr d&2&3&1&4},
$$
and this gives the following refinement of Primc's conjecture.

\begin{theorem}
\label{th:primcrefined}
Let $A_2(n;k,\ell,m)$ denote the number of coloured partitions $\lambda_1 + \cdots + \lambda_s$ of $n$, such that odd parts can be coloured $a$ or $d$ and even parts can be coloured $b$ or $c$, with no part $1_d$, such that $\lambda_i - \lambda_{i+1} \geq B_2(c(\lambda_i),c(\lambda_{i+1}))$, having $k$ parts coloured $a$, $\ell$ parts coloured $c$ and $m$ parts coloured $d$. Then
$$\sum_{n,k,\ell,m \geq 0} A_2(n;k,\ell,m) q^n a^k c^{\ell} d^m = \frac{(-aq;q^4)_{\infty}(-dq^3;q^4)_{\infty}}{(q^2;q^2)_{\infty}(cq^2;q^4)_{\infty}}.$$
\end{theorem}

One recovers Primc's conjecture by setting $a=c=d=1$, as the infinite product becomes
\begin{align*}
\frac{(-q;q^4)_{\infty}(-q^3;q^4)_{\infty}}{(q^2;q^2)_{\infty}(q^2;q^4)_{\infty}} &= \frac{(-q;q^2)_{\infty}(q;q^2)_{\infty}}{(q^2;q^2)_{\infty}(q^2;q^4)_{\infty}(q;q^2)_{\infty}}\\
&=\frac{(q^2;q^4)_{\infty}}{(q;q)_{\infty}(q^2;q^4)_{\infty}}\\
&=\frac{1}{(q;q)_{\infty}}.
\end{align*}

\begin{proof}[Idea of the proof of Theorem~\ref{th:main}]
Define $G_k (q;a,c,d)$ (resp. $E_k (q;a,c,d)$) to be the generating function for coloured partitions satisfying the difference conditions from matrix $A$ with the added condition that the largest part is at most (resp. equal to) $k$.
As usual, we want to find $G_{\infty}(q;a,c,d):= \lim_{k \rightarrow \infty} G_k (q;a,c,d)$, which is the generating function for all partitions with difference conditions, as there is no more restriction on the size of the largest part.

We start by using the matrix $B$ to give four recurrences relating the $G_k (q;a,c,d)$'s and the $E_k (q;a,c,d)$'s. For example, we prove
\begin{align*}
G_{k_{a}}(q;a,c,d)-G_{(k-1)_{d}}(q;a,c,d) &=E_{k_{a}}(q;a,c,d)
\\&=aq^k (E_{(k-1)_{b}}(q;a,c,d)+G_{(k-2)_{d}}(q;a,c,d)).
\end{align*}

Then we combine these four equations to obtain a larger recurrence equation involving only $G_{k_d}(q;a,c,d)$'s :
\begin{equation}
\label{eq:qdiff}
\begin{aligned}
(1-cq^k)&G_{k_{d}}(q;a,c,d)= \frac{1-cq^{2k}}{1-q^k}G_{(k-1)_{d}}(q;a,c,d) 
\\&+ \frac{aq^k+dq^k+adq^{2k}}{1-q^{k-1}}G_{(k-2)_{d}}(q;a,c,d) +\frac{adq^{2k-1}}{1-q^{k-2}}G_{(k-3)_{d}}(q;a,c,d).
\end{aligned}
\end{equation}

Finally we use the technique consisting of going back and forth from recurrences to $q$-difference equations previously introduced by the author~\cite{Doussegene,Doussegene2} to lower the degree of the equations, and conclude by using Appell's lemma to obtain $G_{\infty}(q;a,c,d)$.
\end{proof}
The details of the proof and some applications of Theorem~\ref{th:main} can be found in a paper with Jeremy Lovejoy~\cite{DousseLovejoy}.

\subsection{Andrews' theorems for overpartitions}
The last example of application of our method is a general theorem on coloured overpartitions.

Let $r$ be a positive integer. We define $r$ primary colours $u_1, \dots, u_r$ and use them to define $2^r-1$ colours $\tilde{u}_1, \dots , \tilde{u}_{2^r-1}$ as follows:
$$ \tilde{u}_i := u_1^{\epsilon_1(i)} \cdots u_r^{\epsilon_r(i)},$$
where
$$\epsilon_k(i) := \begin{cases}
1 \text{ if } 2^{k-1} \text{ appears in the binary expansion of } i \\
0 \text{ otherwise.}
\end{cases}$$
They are ordered in the natural ordering, namely
$$\tilde{u}_1 < \cdots < \tilde{u}_{2^r-1}.$$
Now for all $i \in \{1, \dots, 2^r-1 \},$ let $v(\tilde{u}_i)$ (resp. $z(\tilde{u}_i)$) be the smallest (resp. largest) primary colour appearing in the colour $\tilde{u}_i$ and $w(\tilde{u}_i)$ be the number of primary colours appearing in $\tilde{u}_i$. Finally, for $i,j \in \{1, \dots, 2^r-1 \},$ let 
$$\delta(\tilde{u}_i, \tilde{u}_j) :=  \begin{cases}
1 \text{ if } z(\tilde{u}_i) < v(\tilde{u}_j)\\
0 \text{ otherwise.}
\end{cases}$$

We prove the following theorem.
\begin{theorem}
\label{th:andrewsnondil}
Let $\overline{A}(\ell_1, \dots, \ell_r;k,n)$ denote the number of overpartitions of $n$ into non-negative parts coloured $u_1, \dots, u_{r-1}$ or $u_r$, having $\ell_i$ parts coloured $u_i$ for all $i \in \{1, \dots,r\}$ and $k$ non-overlined parts. Let $\overline{B}(\ell_1, \dots, \ell_r;k,n)$ denote the number of overpartitions $\lambda_1+\cdots+ \lambda_s$ of $n$ into non-negative parts coloured $\tilde{u}_1, \dots, \tilde{u}_{2^r-2}$ or $\tilde{u}_{2^r-1}$, such that for all $i \in \{1, \dots,r\}$, $\ell_i$ parts have $u_i$ as one of their primary colours, having $k$ non-overlined parts and satisfying the difference conditions
$$\lambda_i - \lambda_{i+1} \geq w(c(\lambda_{i+1}))+\chi(\overline{\lambda_{i+1}})-1+\delta(c(\lambda_i),c(\lambda_{i+1})),$$
where $\chi(\overline{\lambda_{i+1}}) =1$ if $\lambda_{i+1}$ is overlined and $0$ otherwise.

Then for all $\ell_1, \dots, \ell_r, k, n \geq 0$,
$$\overline{A}(\ell_1, \dots, \ell_r;k,n)= \overline{B}(\ell_1, \dots, \ell_r;k,n).$$
\end{theorem}

Theorem~\ref{th:andrewsnondil} generalises and unifies Schur's theorem~\cite{Schur}, its weighted words version due to Alladi and Gordon~\cite{Alladi} (Theorem~\ref{schurnondil}), its overpartition version due to Lovejoy~\cite{Lovejoy}, two generalisations of Schur's theorem due to Andrews~\cite{Andrews2,Andrews1}, their weighted words version due to Corteel-Lovejoy~\cite{CorLov} and generalisations of Andrews' theorems due to the author\cite{Doussegene,Doussegene2}.
All these generalisations of Schur's theorem are summarised in Figure~\ref{fig:gene}, where $A \longrightarrow B$ means that the theorem corresponding to the infinite product $A$ is generalised by the theorem corresponding to the infinite product $B$.

\begin{figure}
\label{fig:gene}
\caption{The generalisations of Schur's theorem}
\begin{center}
\includegraphics[width=0.9\textwidth]{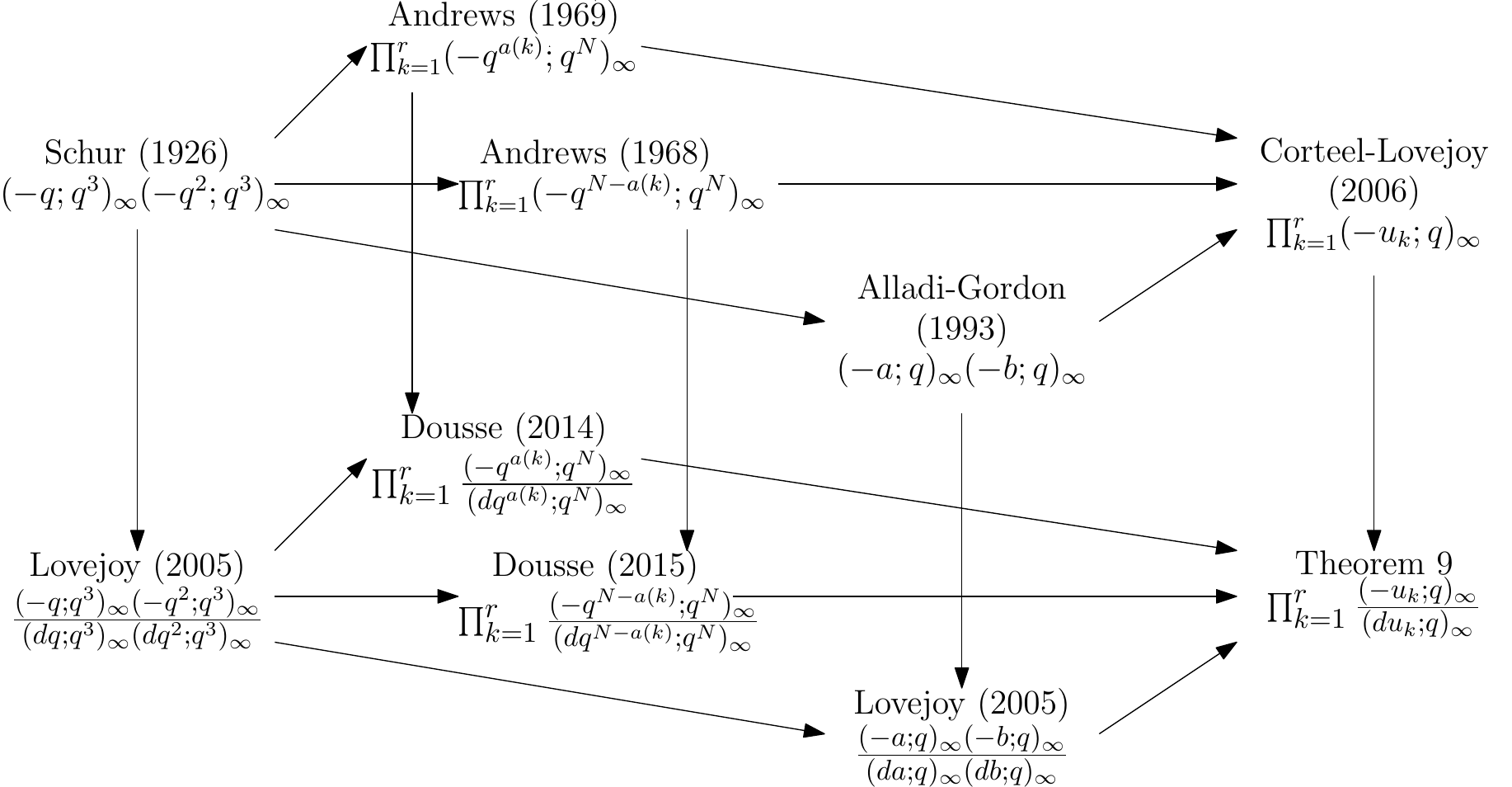}
\end{center}
\end{figure}

Here again, the minimal partitions do not seem easy to obtain as there are many colours and the difference conditions are quite intricate. Therefore our method is easier to apply. However, given that in Theorem~\ref{th:andrewsnondil}, the minimal difference between $\lambda_i$ and $\lambda_{i+1}$ depends on certain conditions on the smaller smaller part $\lambda_{i+1}$, it is more convenient to find $q$-difference equations on the generating functions with an added restriction on the smallest part rather than on the largest part. Details of the proof can be found in the preprint~\cite{Dousseunif}.

\section{Conclusion}
This new version of the method of weighted words seems to have a wide range of application, and it would be interesting to see if it can be used to prove refinements of other partition identities with intricate difference conditions, such as another theorem of Primc~\cite{Primc} which involves difference conditions given by a matrix of size $9 \times 9$.

Another interesting question would be to find representation-theoretic interpretations for the new colour variables added in the identities of Siladi\'c and Primc.

\bibliographystyle{alpha}     

\bibliography{sample}

\end{document}